\documentclass[pdflatex,sn-mathphys-num]{sn-jnl}

\usepackage{graphicx}%
\usepackage{multirow}%
\usepackage{amsmath,amssymb,amsfonts}%
\usepackage{amsthm}%
\usepackage{mathrsfs}%
\usepackage[title]{appendix}%
\usepackage{xcolor}%
\usepackage{textcomp}%
\usepackage{manyfoot}%
\usepackage{booktabs}%
\usepackage{algorithm}%
\usepackage{algorithmicx}%
\usepackage{algpseudocode}%
\usepackage{listings}%

\newcommand{\RM}{\mathbb{R}}
 \newcommand{\bv}{\mathbf{v}}
 \newcommand{\bx}{\mathbf{x}}
 \newcommand{\boldg}{\mathbf{g}}
 \newcommand{\by}{\mathbf{y}} 
 
 \newcommand{\bn}{\mathbf{n}}
 
 \newcommand{\NM}{\mathbb{N}}
 
 \newcommand{\CM}{\mathbb{C}}

 \newcommand{\curl}{\operatorname{curl}}
 \newcommand{\divv}{\operatorname{div}}
 \newcommand{\real}{\operatorname{Re}}
 \newcommand{\imag}{\operatorname{Im}}
 
 \newcommand{\ds}{\operatorname{ds}}

 \newcommand{\sign}{\operatorname{sign}}
 \newcommand{\dx}{\operatorname{d\mathbf{x}}}

\theoremstyle{thmstyleone}%
\newtheorem{theorem}{Theorem}
\theoremstyle{thmstyletwo}%

\theoremstyle{thmstylethree}%

\begin{document}

\title[Exterior 2D Div-Curl Problem]{Explicit Formula of the Infinite Energy Solutions for the Exterior 2D Div-Curl Problem with Dirichlet Boundary Condition}



\author*[1]{\fnm{Aleksei} \sur{Gorshkov}}\email{alexey.gorshkov.msu@gmail.com}

\affil*[1]{\orgdiv{Mechanics and mathematics}, \orgname{Lomonosov Moscow State University}, \orgaddress{\street{Leninskie Gory}, \city{Moscow}, \postcode{119991}, \state{}, \country{Russia}}}

\keywords{Div-curl problem, Hodge-Helmholtz decomposition, no-slip condition, solenoidal flow}


\abstract{In the paper we study the 2D div-curl problem in the exterior domain which models the flow with given vorticity, divergency, boundary condition at infinity, and Dirichlet condition on the solid surface. We will find the relations
on vorticity and divergence for uniqueness solvability of the problem and deduce the explicit formula. }

\maketitle
\section{Introduction}
Let's consider the div-curl problem with Dirichlet boundary condition in 2D exterior domain $\Omega = \RM^2 \setminus G$:
\begin{eqnarray}
&&{\rm div}~ \bv(\bx) = \rho(\bx), \label{intro:1}\\
&&{\rm curl}~ \bv(\bx) = w(\bx), \label{intro:2}\\
&&\bv(\bx)=\boldg(\bx),~\bx\in\partial G,\label{intro:3}\\
&&\bv(\bx)\to \bv_\infty,~|\bx|\to \infty. \label{intro:4}
\end{eqnarray}

Here $\bx=(x_1,x_2) \in \Omega \subset \RM^2$, $\Omega = \RM^2\setminus \overline G$ is the exterior to bounded simply connected domain $G$ with a Lipschitz boundary,  $\bv = (v_1,v_2)$ - 2D vector field, ${\rm curl}~\bv(\bx) = \partial_{x_1}v_2 - \partial_{x_2}v_1$ is the vorticity, ${\rm div}~ \bv(\bx) = \partial_{x_1}v_1 + \partial_{x_2}v_2$ - the divergence, $\boldg(\bx)$ - given boundary condition at $\partial G$,  $\bv_\infty \in \RM^2$ - is the given constant velocity of the  flow at infinity.

The div-curl problem was the subject of many studies. In paper \cite{KirchhartSchulz} this problem was studied in bounded domain for the solenoidal flow with $w \in H^{-1}$. The space $H^{-1}$  forces the use of singular space $H^{-1/2}(\partial G)$ for tangential boundary condition. 
We study the case of square-integrable functions  $\rho(\bx)$, $w(\bx)$ with $\boldg(\bx) \in H^{1/2}(\partial G)$. In the bounded domains for flows with finite energy, this problem also was studied in \cite{Divcurl1}\cite{Divcurl2}\cite{Divcurl3}. In \cite{G1} author solved the exterior div-curl problem for solenoidal flows with no-slip boundary condition.

In the paper we will construct the explicit formula of the solution and derive several estimates including the following ones:
\begin{align}\label{nablaest}
\|\nabla \bv (\bx)\|_{L_2(\Omega)} \leq C \left ( \|\rho (\bx)\|_{L_2(\Omega)} +
\|w (\bx)\|_{L_2(\Omega)} + \|\boldg(\bx)\|_{H^{1/2}(\partial \Omega)} \right ).
\end{align}
Here $\partial \Omega$ denotes the inner boundary of the exterior domain $\Omega$

The $L_2$-estimate of the velocity field $\bv$ is an intriguing problem. Even when ones subtract $\bv_\infty$, the velocity field $\bv - \bv_\infty$ may still have infinite $L_2$-norm due to the existence of the harmonic fields with infinite kinetic energy\cite{G1}. These fields decay at infinity like $1/|\bx|$. Zero circularity and zero flux of the flow at infinity annihilate these fields with infinite energy allowing to obtain finite $L_2$-norms of $\bv - \bv_\infty$.

Our aim is to find restrictions on $w$, $\rho$ for solvability of the div-curl problem (\ref{intro:1})-(\ref{intro:4}) and to prove the estimate (\ref{nablaest}). 
Quartapelle and Valz-Gris in paper \cite{Quartapelle2} derived the {\it orthogonality conditions} on the vorticity for solenoidal flows defined in bounded domains. We will find a new criteria for solvability of the above div-curl problem in the case of non-solenoidal exterior flows. For example, for the flow around the disk with radius $r_0$ when $\Omega =\{\bx \in \RM^2,~|\bx| > r_0 \}$  with given constant flow at infinity $\bv_\infty$ these conditions will have the form:
\begin{align}
& \int_{\Omega} \frac {w(\bx)+i\rho(\bx)}{(x_1+ix_2)^k} d\bx + \oint_{\partial \Omega} \frac{\boldg(\bx)\cdot d\mathbf{l} +  i(\boldg(\bx),\bn) dl }{(x_1+ix_2)^k} = \oint_{\partial \Omega} \frac{\bv_\infty\cdot d\mathbf{l} + i(\bv_\infty,\bn)dl}{(x_1+ix_2)^k},~k \in \NM.\label{intcond}
\end{align}
Here $\oint \cdot d\mathbf{l}$ is a line integral of a vector field, and $\oint \cdot dl$ is a line integral of a scalar field. For constant vector fields $\bv_\infty$ its Fourier series expansion by angle involves only the terms with $k=\pm 1$. So the integration in right-hand side of (\ref{intcond}) over the inner boundary $\oint_{\partial \Omega}$ can be changed to integration over the exterior boundary $\lim_{R\to\infty}\oint_{\vert\bx\vert=R}$. If $\rho =\boldg=\bv_\infty=0$ then these relations are the condition of the orthogonality $w$ to harmonic functions $1/(x_1+ix_2)^k$.

Since the right-hand side of (\ref{intcond}) is equal to 0 for $k=0$ then from Stokes' and Green's formulas follow that with the non-zero circulation or flux at infinity (which generate harmonic fields with infinite kinetic energy), we will have 
\begin{align*}
\int_{\Omega}\left ( w(\bx)+i\rho(\bx) \right )d\bx  + \oint_{\partial \Omega} \boldg(\bx) \cdot d\mathbf{l}
+ i\oint_{\partial \Omega} (\boldg(\bx),\bn)dl \neq 0. 
\end{align*}

In order to obtain finite $L_2$-norm $\|\bv(\cdot) - \bv_\infty\|_{L_2}$ one should additionally impose restrictions on circulation and flux:
\begin{align} \label{intro:zerocirculation}
&\lim_{R\to\infty}\oint_{\vert\bx\vert=R} \bv \cdot d\mathbf{l} = 0,\\
&\lim_{R\to\infty}\oint_{\vert\bx\vert=R} (\bv,\bn) dl = 0, \label{intro:zeroflux}
\end{align}  
or
\begin{align}\label{zerofluxandcirc} \int_{\Omega}\left ( w(\bx)+i\rho(\bx) \right )d\bx  + \oint_{\partial \Omega} \boldg(\bx) \cdot d\mathbf{l}
+ i\oint_{\partial \Omega} (\boldg(\bx),\bn)dl = 0. 
\end{align}

So, for $L_2$-estimate of $\bv(\cdot) - \bv_\infty$ the relations (\ref{intcond}) also must be valid and for $k=0$. For $L_p$-estimate  with $p>2$ the above condition is not necessary. 

For the $L_2$-estimate in the paper we will use weighted space
\begin{align*}L_{2,N}(\Omega) = \{f(x) : 
\|f(\cdot)\|_{L_{2,N}(\Omega)}^2=\int_\Omega |f(x)|^2 (1+|x|^2)^N \mathrm{d}x <\infty\}.\end{align*}

Besides (\ref{nablaest}) we will obtain estimates of the vector field in $L_2$, $L_\infty$, and in the Sobolev space $H^1$. $L_\infty$-estimates play an important role in proving the existence theorems in the non-linear fluid dynamics. It helps to estimate nonlinear terms like $(\bv,\nabla)\bv$, $(\bv,\nabla w)$ in the fixed-point iteration method. 

\section{Div-curl problem in the exterior of the disk}

First, we study (\ref{intro:1})-(\ref{intro:4}) in $B_{r_0}=\{\bx \in \RM^2,~|\bx| > r_0 \},~r_0>0$ with boundary $S_{r_0}=\{\bx \in \RM^2,~|\bx| = r_0 \}$. We construct the solution of (\ref{intro:1}) - (\ref{intro:zerocirculation}) in $B_{r_0}$ as a Fourier series
in polar coordinates $r$, $\varphi$:
\begin{eqnarray}
 &&\bv(r,\varphi) = \sum_{k=-\infty}^\infty \bv_{k}(r)e^{ik\varphi},\nonumber \\ 
  &&\boldg(r,\varphi) = \sum_{k=-\infty}^\infty \boldg_{k}(r)e^{ik\varphi},\nonumber \\ 
 &&\rho(r,\varphi) = \sum_{k=-\infty}^\infty \rho_k(r)e^{ik\varphi}\nonumber, \\
 &&w(r,\varphi) = \sum_{k=-\infty}^\infty w_k(r)e^{ik\varphi}\nonumber.
\end{eqnarray}

Here $\bv(r,\varphi) = (v_r, v_\varphi)$, $\boldg(r,\varphi) = (g_r, g_\varphi)$ - vector fields in polar coordinates, $\bv_k(r) = (v_{r,k}, v_{\varphi,k})$, $\boldg_k(r) = (g_{r,k}, g_{\varphi,k})$, $\rho_k$, $w_k$,  - are the Fourier coefficients.

The Fourier coefficients for exterior flow $\bv_\infty$ are given by the formulas:
\begin{align*}
&v^\infty_{r,k}=\frac{\delta_{\vert k \vert,1}}2 (v^\infty_1 - i k v^\infty_2), \\ 
&v^\infty_{\varphi,k}=\frac{\delta_{\vert k \vert,1}}2 (v^\infty_2 + i k v^\infty_1),
\end{align*} 
and so we have the relation between $v^\infty_{\varphi,k}$ and $iv^\infty_{r,k}$:
\begin{align*} 
v^\infty_{\varphi,k} =  \sign(k) iv^\infty_{r,k}.	
\end{align*}
Here $\delta_{\vert k \vert,1}={\begin{cases}1,\vert k \vert=1,\\0,\vert k \vert \neq 1.\end{cases}}$, $\sign(k) = \begin{cases} 1,~k > 0 \\ 0,~k=0 \\ -1,~k<0 \end{cases}$.

All the Fourier coefficients of the external flow are zero (except $k=\pm 1$). From that follows, that right-hand side in (\ref{intcond}) is equal to zero for all $k\in \NM$ except $k=1$. For horizontal flow $\bv_\infty=(v_\infty,0)$
\begin{align*}
&v^\infty_{r,k}=\frac{\delta_{\vert k \vert,1}}2 v_\infty, \\ 
&v^\infty_{\varphi,k}=i k\frac{\delta_{\vert k \vert,1}}2  v_\infty.
\end{align*}

Equations (\ref{intro:1}), (\ref{intro:2}) in polar coordinates are written as:
\begin{eqnarray*}
&&{\frac {1}{r}}{\frac {\partial }{\partial r}}\left(rv_{r,k}\right)+{\frac {ik}{r}} v_{\varphi,k} = \rho_k,\\
&&{\frac {1}{r}}{\frac {\partial }{\partial r}}\left(rv_{\varphi,k}\right)-{\frac {ik}{r}} v_{r,k} = w_k.
\end{eqnarray*}

Further considerations will be held only for non-negative Fourier coefficient indices $k$. This is because in the case of the real-valued functions $w$, $\rho$, $\boldg$ the Fourier coefficients for negative $k$ satisfy $f_{-k} = \overline{f_k}$ and can be estimated by the same way as for positive $k$. But even in the case of complex-valued functions the subsequent proofs can be easily adopted on negative values of $k$.

For $k \in \NM$ the solution of this system with condition at infinity (\ref{intro:4}) is given by the following formulas:
\begin{eqnarray} \label{intro:BiotSavar1}
&&v_{r,k} =  \frac{ir^{-k-1}}2 \int_{r_0}^r s^{k+1} \left (   w_k(s) - i \rho_k(s) \right)\ds \\ 
&&+   \frac{ir^{k-1}}2 \int_r^\infty s^{-k+1} \left(    w_k(s) + i \rho_k(s) \right)\ds  +  \alpha_k i r^{-k-1} + v^\infty_{r,k}, \nonumber \\ \label{intro:BiotSavar2}
&&v_{\varphi,k} = \frac{r^{-k-1}}2 \int_{r_0}^r s^{k+1}\left (w_k(s) - i \rho_k(s) \right)\ds \\ 
&&- \frac{r^{k-1}}2 \int_r^\infty s^{-k+1}\left(w_k(s) + i \rho_k(s) \right)\ds + \alpha_k   r^{-k-1} + v^\infty_{\varphi,k}.
\nonumber \end{eqnarray}

Multiplying (\ref{intro:BiotSavar2}) by $i$, then substituting $r=r_0$, and after the adding and the subtracting of equations (\ref{intro:BiotSavar1}), (\ref{intro:BiotSavar2}), we will obtain 
$$
\alpha_k = \frac{r_0^{k+1}}{2}\left (g_{\varphi,k}  - i g_{r,k} \right )
$$
and
\begin{align*}
r_0^{k-1} \int_{r_0}^\infty s^{-k+1} \left(    w_k(s) + i \rho_k(s) \right)\ds +  g_{\varphi,k} +i g_{r,k}=v^\infty_{\varphi,k}  + i v^\infty_{r,k}.
\end{align*}


The last formula is equivalent to relations (\ref{intcond}) for $k\in \NM$. This is the necessary condition for solvalibility of the Dirichlet div-curl problem. In the next section we extend this result to exterior domain by means of the Riemann mapping theorem.

For $k=0$ from the condition (\ref{intro:3}) we have:
\begin{align*}
&v_{r,0} = \frac 1r \int_{r_0}^r s \rho_0(s)\ds +g_{r,0}/r, \nonumber\\
&v_{\varphi,0} = \frac 1r \int_{r_0}^r s w_0(s)\ds +g_{\varphi,0}/r \nonumber.
\end{align*}

In order to get estimates of $\|\bv - \bv_\infty\|_{L_2}$ we additionally impose zero-circulation (\ref{intro:zerocirculation}) and zero-flux (\ref{intro:zeroflux}) conditions. Then from Stokes' and Green's formulas
\begin{align*}
&\int_{r_0}^\infty s \rho(s)\ds+g_{r,0}=0, \\
&\int_{r_0}^\infty s w(s)\ds+g_{\varphi,0}=0,
\end{align*}

These relations are equivalent to (\ref{zerofluxandcirc}). 

Formulas (\ref{intro:BiotSavar1}), (\ref{intro:BiotSavar2}) have an integral representation with singular kernels: 
\begin{align}
\bv(\bx) =\frac 1{2\pi} \int_{B_{r_0}} \frac{\bx-\by}{\vert\bx-\by\vert^2} \rho(\by) \operatorname{d\by} +
\frac 1{2\pi} \int_{B_{r_0}} \frac{(\bx-\by)^\perp}{\vert\bx-\by\vert^2} w(\by) \operatorname{d\by} + \nonumber \\
\frac 1{2\pi} \oint_{S_{r_0}} \frac{\bx-\by}{\vert\bx-\by\vert^2} (\boldg,\bn)dl +
\frac 1{2\pi} \oint_{S_{r_0}} \frac{(\bx-\by)^\perp}{\vert\bx-\by\vert^2} (\boldg,\tau)dl + \bv_\infty,\label{intro:BiotSavarInt}
\end{align}
where $\bx^\perp = (-x_2,x_1)$. Here $\tau$ is a unit tangent vector to $S_{r_0}$ in a counter-clockwise direction.

The kernels in the formula above are the gradient $\nabla_\bx$ and the skew gradient $\nabla_\bx^\perp = 
(-\partial_{x_2}, \partial_{x_1})$ of the Green function
\begin{align*}
G(\bx,\by) = \frac 1{2\pi} \ln |\bx-\by|.
\end{align*}

In particular, the formulas (\ref{intro:BiotSavar1}), (\ref{intro:BiotSavar2}) with no-slip condition (\ref{intro:3}) when $\boldg=0$ lead to relations: 
\begin{equation} \label{intro:noslipcondintegral}
\int_{r_0}^\infty s^{-k+1} \left(w_k(s) + i \rho_k(s) \right) \ds = v^\infty_{\varphi,k}  + i v^\infty_{r,k}. 
\end{equation}

These relations are the orthogonality conditions
\begin{align*}
\int_{B_{r_0}} \frac {w(\bx)+i\rho(\bx)}{(x_1+ix_2)^k} d\bx =
\oint_{S_{r_0}} \frac{\bv_\infty\cdot d\mathbf{l}}{(x_1+ix_2)^k} 
+ \nonumber \\ i\oint_{S_{r_0}} \frac{(\bv_\infty,\bn)}{(x_1+ix_2)^k} dl,~k \in \NM \cup \{0\}. 
\end{align*}

\begin{theorem}\label{L2estthmdisk}[$L_2$-estimate of the $\nabla \bv$]  Let $\Omega = B_{r_0}$ be an exterior of the disk, $\rho, w \in L_{2}(B_{r_0})$ and the relations (\ref{intcond}) are satisfied for $k\in \NM \cup \{0\}$. Then the solution $\bv$ of the problem (\ref{intro:1})-(\ref{intro:4}) is given by explicit formula (\ref{intro:BiotSavarInt}) and satisfies the estimate (\ref{nablaest}). 
\end{theorem}

\begin{proof}

The formulas (\ref{intro:BiotSavar1}), (\ref{intro:BiotSavar2}) decompose the velocity on the three components
$$
\bv = \bv_1 + \bv_2 + \bv_\infty,
$$
where $\bv_1$ depends on $w$ and $\rho$, when $\bv_2$ depends on boundary condition $\boldg$.

The singular integrals involved in formula (\ref{intro:BiotSavarInt}) are the operators of Calder\'on-Zygmund type (see \cite{Calderon}) and for $p >1$ (and in particular for $p=2$)  
\begin{align}\label{bsest2}
\|\nabla \bv_1(\cdot)\|_{L_p} \leq C \left ( \| \rho \|_{L_p} + \| w \|_{L_p} \right ). 
\end{align} 

From the Hardy-Littlewood-Sobolev inequality for $1<p<2<q<\infty$, satisfying $\frac 1q = \frac 1p - \frac 12$, the following estimate holds with some $C>0$ (see \cite{Stein}):
\begin{equation}\label{intro:LpLqest}
 \|\bv_1 - \bv_\infty\|_{L_q} \leq C
\left ( \|\rho \|_{L_p} + \| w \|_{L_p} \right ) \nonumber.
\end{equation}


Now we estimate the rest term in (\ref{intro:BiotSavar1}), (\ref{intro:BiotSavar2}), which defines $\bv_2$ and involves the boundary function $\boldg$. We will use the space $L_{2}(r_0,\infty; r)$, which is associated with $L_{2}(B_{r_0})$, and is defined as
\begin{align*}L_{2}(r_0,\infty; r) = \{f(r) : 
\|f(\cdot)\|_{L_{2}(r_0,\infty; r)}^2=\int_{r_0}^\infty |f(r)|^2 r \mathrm{d}r <\infty\}.
\end{align*}

The gradient $\nabla \bv_2$ involves such terms as $\frac d{dr}v_{r,k}$, $\frac d{dr}v_{\varphi,k}$, $ikv_{r,k}$, $ikv_{\varphi,k}$. Then we have
$$
\left \| \frac d{dr} \left (\alpha_k r^{-k-1} \right ) \right \|^2_{L_2(r_0,\infty; r)} = \frac {|\alpha_k|^2 (k+1)}{2r_0^{2k+2}} \leq  \frac {k+1} 4 \left (|g_{\varphi,k}|^2 + |g_{r,k}|^2 \right )
$$
and
$$
\left \| ik \alpha_k r^{-k-1}  \right \|^2_{L_2(r_0,\infty; r)} = \frac {|\alpha_k|^2 k}{2r_0^{2k}} \leq  \frac {k r_0^2} 4 \left (|g_{\varphi,k}|^2 + |g_{r,k}|^2 \right ).
$$

Define $H^{1/2}(S_{r_0})$ as the completion of smooth functions by the norm:
$$
\|\boldg(\bx)\|^2_{H^{1/2}(S_{r_0})} = \sum_{k=-\infty}^\infty (|k|+1) \left (|g_{\varphi,k}|^2 + |g_{r,k}|^2 \right ).
$$

Then 
$$
\sum_{k=-\infty}^\infty \|\alpha_k r^{-|k|+1}\|^2_{L_2(r_0,\infty; r)} \leq    C \|\boldg(\bx)\|^2_{H^{1/2}(S_{r_0})}
$$
with some $C=C(r_0)$.
\end{proof}

Now we are ready to derive the energy estimate.

\begin{theorem}[$L_2$-estimates of the velocity]\label{intro:propl2est}
Suppose that the orthogonality conditions (\ref{intcond}) are satisfied for $k\in \NM \cup \{0\}$ and $\rho, w\in L_{2,N}(B_{r_0})$ with $N>1$. Then there exists the unique solution of the div-curl problem (\ref{intro:1}) - (\ref{intro:4}) in the exterior of the disk $B_{r_0}$, which is given by explicit formula (\ref{intro:BiotSavarInt}), and with some $C=C(r_0,N)$ the following estimate holds:
\begin{align} \label{vesinh1t}
\| \bv(\cdot) - \bv_\infty \|_{H^1(B_{r_0})} \leq C \left (\| \rho \|_{L_{2,N}(B_{r_0})} + \| w \|_{L_{2,N}(B_{r_0})}  + \|\boldg(\bx)\|_{H^{1/2}(S_{r_0})}\right ).
\end{align} 
\end{theorem}

\begin{proof}
Note, that from (\ref{intcond}) with $k=0$ follows that the conditions (\ref{intro:zerocirculation}), (\ref{intro:zeroflux}) are satisfied. Existence follows from the explicit formulas (\ref{intro:BiotSavar1}), (\ref{intro:BiotSavar2}), which have an integral repsentation  (\ref{intro:BiotSavarInt}). Uniqueness follows from the Hodge theory since the circularity and the flux around circle are fixed.

Now we prove (\ref{vesinh1t}). We estimate the first term in (\ref{intro:BiotSavar1}), (\ref{intro:BiotSavar2}) involving $w$ for $k \neq 0$, $k \neq N-1$:
\begin{align*}
&\left | r^{-\vert k \vert-1} \int_{r_0}^r s^{\vert k \vert+1}w_k(s)\ds \right | =  \left | r^{-\vert k \vert-1} \int_{r_0}^r \frac{s^{\vert k \vert+\frac 12}}{(1+s)^N}w_k(s) (1+s)^N \sqrt s\ds \right |\leq 
\\
&\left | r^{-\vert k \vert-1} \int_{r_0}^r s^{\vert k \vert+\frac 12 - N}w_k(s) (1+s)^N \sqrt s\ds \right |\leq 
\\
&r^{-\vert k \vert-1} \sqrt{\int_{r_0}^r s^{\vert 2k \vert+1 - 2N}\ds} 
  \sqrt{\int_{r_0}^r w^2_k(s) (1+s^2)^N s\ds}  \leq 
  \\
  & \frac C{\sqrt{|2k-2N+2|}}
 \|w_k(\cdot)\|_{L_{2,N}(r_0,\infty; r)} \left ( \frac 1 {r^N} + \left ( \frac{r_0}{r} \right )^{k+1} \cdot  \frac 1{r_0^N} \right ).
\end{align*}
Here $L_{2,N}(r_0,\infty; r)$ is the weighted space of functions (with the weight $(1+|r|^2)^N r$) which is related to the space $L_{2,N}(B_{r_0})$:
\begin{align*}L_{2,N}(r_0,\infty; r) = \{f(r) : 
\|f(\cdot)\|_{L_{2,N}(r_0,\infty; r)}^2=\int_{r_0}^\infty |f(r)|^2 (1+|r|^2)^N r \mathrm{d}r <\infty\}.\end{align*}

Then we have $L_2$-estimate:
\begin{align*}
&\int_{r_0}^\infty \left | r^{-\vert k \vert-1} \int_{r_0}^r s^{\vert k \vert+1}w_k(s)\ds \right |^2 rdr \leq 
  \\
  & \int_{r_0}^\infty \frac {C^2}{|2k-2N+2|}
 \|w_k(\cdot)\|^2_{L_{2,N}(r_0,\infty; r)} \left ( \frac 1 {r^N} + \left ( \frac{r_0}{r} \right )^{k+1} \cdot  \frac 1{r_0^N} \right )^2rdr \leq \\
 &\frac {C^2}{|k-N+1|}
 \|w_k(\cdot)\|^2_{L_{2,N}(r_0,\infty; r)} \left ( \frac {r_0^{-2N+2}}{2N-2} + \frac{r_0^{2-2N}}{2k} \right ).
\end{align*}

For $k=N-1$:
\begin{align*}
&\left | r^{-\vert k \vert-1} \int_{r_0}^r s^{\vert k \vert+1}w_k(s)\ds \right | =  \left | r^{-\vert k \vert-1} \int_{r_0}^r \frac{s^{\vert k \vert+\frac 12}}{(1+s)^N}w_k(s) (1+s)^N \sqrt s\ds \right |\leq 
\\
&\left | r^{-\vert k \vert-1} \int_{r_0}^r s^{\vert k \vert+\frac 12 - N}w_k(s) (1+s)^N \sqrt s\ds \right |\leq 
\\
&r^{-\vert k \vert-1} \sqrt{\int_{r_0}^r s^{-1}\ds} \|w_k(\cdot)\|_{L_{2,N}(r_0,\infty; r)}   \leq 
   \frac{ \sqrt {\ln r - \ln r_0}}{r^{k+1}}  \|w_k(\cdot)\|_{L_{2,N}(r_0,\infty; r)}.
\end{align*}

A similar estimate is valid for $\rho$. And so, the first term in (\ref{intro:BiotSavar1}), (\ref{intro:BiotSavar2}) for  $k\neq 0$ belongs to $L_2(r_0,\infty; r)$.

Second term in (\ref{intro:BiotSavar1}), (\ref{intro:BiotSavar2}) we estimate in a similar way: 
\begin{align*}
&\left | r^{\vert k \vert-1} \int_r^\infty s^{-\vert k \vert+1}w_k(s)\ds \right | =  \left | r^{\vert k \vert-1} \int_r^\infty \frac{s^{-\vert k \vert+\frac 12}}{(1+s)^N}w_k(s) (1+s)^N \sqrt s\ds \right |\leq 
\\
&\left | r^{\vert k \vert-1} \int_r^\infty s^{-\vert k \vert+\frac 12 - N}w_k(s) (1+s)^N \sqrt s\ds \right |\leq 
\\
&r^{\vert k \vert-1} \sqrt{\int_r^\infty s^{-\vert 2k \vert+1 - 2N}\ds} 
  \sqrt{\int_r^\infty w^2_k(s) (1+s^2)^N s\ds}  \leq 
  \\
  & \frac 1{r^N \sqrt{|2k+2N-2|}}
 \|w_k(\cdot)\|_{L_{2,N}(r_0,\infty; r)}. 
\end{align*}

Now for $k\neq 0$ we estimate the term in (\ref{intro:BiotSavar1}), (\ref{intro:BiotSavar2}), which involves $\boldg$. We will have
$$
\left \| \alpha_k r^{-k-1} \right \|^2_{L_2(r_0,\infty; r)} = \frac {|\alpha_k|^2 }{2kr_0^{2k}} \leq  \frac {r_0^2} {4 k} \left (|g_{\varphi,k}|^2 + |g_{r,k}|^2 \right ).
$$
Summarizing by $k \neq 0$ we obtain $\|\boldg(\bx)\|^2_{H^{-1/2}(S_{r_0})}$ in right hand side which is majored by $H^{1/2}$ norm.

So, for $k\neq 0$ it also belongs to $L_2(r_0,\infty; r)$. Lets study the case $k=0$. We have
\begin{align*}
v_{\varphi,0} = \frac 1r \int_{r_0}^r s w_0(s)\ds   + \frac{g_{\varphi,0}}r = -\frac 1r \int_r^\infty s w_0(s)\ds,
\end{align*}
and
\begin{align*}
&|v_{\varphi,0}| =  \left | \frac 1r \int_r^\infty s w_0(s)\ds \right | \leq \\
&\left | r^{-1} \int_r^\infty s^{\frac 12 - N}w_0(s) (1+s)^N \sqrt s\ds \right |\leq 
\\
&r^{-1} \sqrt{\int_r^\infty s^{1 - 2N}\ds} 
  \sqrt{\int_r^\infty w^2_0(s) (1+s^2)^N s\ds}  \leq 
  \\
  & \frac 1{r^N \sqrt{2N-2}}
 \|w_0(\cdot)\|_{L_{2,N}(r_0,\infty; r)}. 
\end{align*}

In a similar way we will have
\begin{align*}
|v_{r,0}|   \leq  \frac 1{r^N \sqrt{2N-2}}
 \|\rho_0(\cdot)\|_{L_{2,N}(r_0,\infty; r)}. 
\end{align*}

So, $v_{r,0}$, $v_{\varphi,0}$ will belong to $L_2(r_0,\infty; r)$. Then, summing by $k$ the estimates of $v_{r,k}$, $v_{\varphi,k}$, given in (\ref{intro:BiotSavar1}), (\ref{intro:BiotSavar2}), and using Parseval's equality for Fourier series, we obtain the inequality with some $C=C(r_0,N)$: 
\begin{align*}
\| \bv(\cdot) - \bv_\infty \|_{L_2(B_{r_0})}\leq C \left (\| \rho \|_{L_{2,N}(B_{r_0})} + \| w \|_{L_{2,N}(B_{r_0})} + \|\boldg(\bx)\|_{H^{-1/2}(S_{r_0})}\right ).
\end{align*}

Finally, in virtue of the inequality (\ref{nablaest}) we obtain the required esti\-mate (\ref{vesinh1t}).
\end{proof}

\section{Div-curl problem in the exterior planar domain}For further consideration we need some additional requirements on $\Omega$. Suppose that there exists a Riemann mapping $\Phi$ from $\Omega$ into $B_{r_0}$ such that in complex form it can be written as
\begin{align*}
\Phi(p)=p+O\left (\frac 1p \right ),
\end{align*}
where $p=y_1+iy_2 \in \CM$. 

We will use $\Phi$ as a mapping from $\CM$ to $\CM$ as well as from $\RM^2$ to $\RM^2$. And also suppose that for $z=x_1+ix_2 \in \CM$ there exists the inverse transform $\Phi^{-1}(z) : B_{r_0} \to \Omega$, which satisfies
\begin{align}\label{intro:Phiprop}
\Phi^{-1}(z)=z+O\left (\frac 1z \right ),
\end{align}
and
\begin{align}\label{intro:Phiprop2}
\left(\Phi^{-1}\right )'(z)=1+O\left (\frac 1{z^2} \right ).
\end{align}

We change the variables $\by=(y_1, y_2) \in \Omega$ onto $\bx=(x_1, x_2) \in B_{r_0}$ in the system (\ref{intro:1}), (\ref{intro:2}). Vector field $\bv(p)$ in $\Omega$ with $p=y_1+iy_2$ defines the vector field $\hat \bv(\bx)$ in $B_{r_0}$: 
\begin{align*}
\bv(\by) =  D\Phi^t \hat \bv(\bx(\by)),
\end{align*}
where $D\Phi^t$ is the transpose matrix to Jacobian matrix $D\Phi$
\begin{align*}
D\Phi = \begin{pmatrix}
&\real \Phi'_{x_1} &\real \Phi'_{x_2} \\ &\imag \Phi'_{x_1} &\imag \Phi'_{x_2}
\end{pmatrix} = \begin{pmatrix}
&\real \Phi'_{x_1} &-\imag \Phi'_{x_1} \\ &\imag \Phi'_{x_1} &\real \Phi'_{x_1}
\end{pmatrix}.
\end{align*}

Then
\begin{align}\label{intro:covariantv}
&\hat \bv(\bx)={D\Phi^{-1}}^t \bv(\Phi^{-1}(\bx))=\\&{D\Phi^{-1}}^t \bv(\real \Phi^{-1}(x_1+ix_2), \imag \Phi^{-1}(x_1+ix_2)) \nonumber
\end{align}
is the vector field in $B_{r_0}$. And $\hat \rho(\bx) = \rho(\Phi^{-1}(\bx))$, $\hat w(\bx) = w(\Phi^{-1}(\bx))$ are the divergence  and vorticity functions in $B_{r_0}$ correspondingly, when 
$$\hat \boldg(\bx) = {D\Phi^{-1}}^t \boldg(\Phi^{-1}(\bx))$$ 
defines the boundary condition (\ref{intro:3}) on the circle $S_{r_0}$.

From the Cauchy-Riemann relations
\begin{align*}
&\frac{\partial x_1}{\partial y_1} = \frac{\partial x_2}{\partial y_2}, \\
&\frac{\partial x_2}{\partial y_1} = - \frac{\partial x_1}{\partial y_2}
\end{align*}
we will have
\begin{align*}
\divv_{\by} \bv(\by) = \frac{\partial v_1}{\partial y_1}+\frac{\partial v_2}{\partial y_2}=
 \frac{\partial v_1}{\partial x_1} \frac{\partial x_1}{\partial y_1}+
 \frac{\partial v_1}{\partial x_2} \frac{\partial x_2}{\partial y_1}+\frac{\partial v_2}{\partial x_1} \frac{\partial x_1}{\partial y_2}+
 \frac{\partial v_2}{\partial x_2} \frac{\partial x_2}{\partial y_2}\\
 =\frac{\partial v_1}{\partial x_1} \frac{\partial x_1}{\partial y_1}+
 \frac{\partial v_1}{\partial x_2} \frac{\partial x_2}{\partial y_1}+\frac{\partial v_2}{\partial x_1} \frac{\partial x_1}{\partial y_2}+
 \frac{\partial v_2}{\partial x_2} \frac{\partial x_2}{\partial y_2} =\frac{\partial x_1}{\partial y_1}\left ( \frac{\partial v_1}{\partial x_1}+\frac{\partial v_2}{\partial x_2} \right )\\-
 \frac{\partial x_2}{\partial y_1} \left (\frac{\partial v_2}{\partial x_1}  - \frac{\partial v_1}{\partial x_2} \right )=
 \frac{\partial x_1}{\partial y_1} \divv_{\bx} \bv(\bx) -
 \frac{\partial x_2}{\partial y_1} \curl_{\bx} \bv(\bx),
\end{align*}
and
\begin{align*}
\curl_{\by} \bv(\by) = \frac{\partial v_2}{\partial y_1} - \frac{\partial v_1}{\partial y_2}=
 \frac{\partial v_2}{\partial x_1} \frac{\partial x_1}{\partial y_1}+
 \frac{\partial v_2}{\partial x_2} \frac{\partial x_2}{\partial y_1}-\frac{\partial v_1}{\partial x_1} \frac{\partial x_1}{\partial y_2}-
 \frac{\partial v_1}{\partial x_2} \frac{\partial x_2}{\partial y_2}\\
 =
 \frac{\partial x_1}{\partial y_1} \curl_{\bx} \bv(\bx) +
 \frac{\partial x_2}{\partial y_1} \divv_{\bx} \bv(\bx). 
\end{align*}

Finally, the system (\ref{intro:1}), (\ref{intro:2}) goes to
\begin{align*}
&\frac{\partial x_1}{\partial y_1} \divv_{\bx} \bv(\bx)- 
\frac{\partial x_2}{\partial y_1} \curl_{\bx} \bv(\bx)= \hat \rho(\bx), \\
&\frac{\partial x_2}{\partial y_1} \divv_{\bx} \bv(\bx) + 
\frac{\partial x_1}{\partial y_1} \curl_{\bx} \bv(\bx) = \hat w(\bx). 
\end{align*}

Then
\begin{align*}
\divv_{\bx} \bv(\bx) = 
 \frac {\frac{\partial x_2}{\partial y_2} }{
\left ( \frac{\partial x_1}{\partial y_1} \right )^2 +\left (\frac{\partial x_1}{\partial y_2} \right )^2}\hat \rho(\bx) +
\frac {\frac{\partial x_2}{\partial y_1} }{
\left ( \frac{\partial x_1}{\partial y_1} \right )^2 +\left (\frac{\partial x_1}{\partial y_2} \right )^2}\hat w(\bx), \\
\curl_{\bx} \bv(\bx) = -\frac {\frac{\partial x_2}{\partial y_1} }{
\left ( \frac{\partial x_1}{\partial y_1} \right )^2 +\left (\frac{\partial x_1}{\partial y_2} \right )^2}\hat \rho(\bx)
+ \frac {\frac{\partial x_1}{\partial y_1} }{
\left ( \frac{\partial x_1}{\partial y_1} \right )^2 +\left (\frac{\partial x_1}{\partial y_2} \right )^2}\hat w(\bx).
\end{align*}

Using the representation of the complex derivative $\Phi'(p)$:
\begin{align*}
\Phi'(p) = \frac{\partial x_1}{\partial y_1} + i \frac{\partial x_2}{\partial y_1}  = \frac{\partial x_2}{\partial y_2} - i \frac{\partial x_1}{\partial y_2}
\end{align*}
we will have
\begin{equation*}
\curl \bv_{\bx}  + i \divv \bv_{\bx} = \frac {\Phi'(p)}{\vert \Phi'(p)\vert ^2} \hat w +
i \overline{ \frac {\Phi'(p)}{\vert \Phi'(p)\vert ^2} } \hat \rho.
\end{equation*}

Now we calculate the divergence of the field $\hat \bv = {D\Phi^{-1}}^t \bv(\Phi^{-1}(\bx))$. Using Jacobian matrix
\begin{align*}
D\Phi^{-1} = \begin{pmatrix}
&\frac{\partial y_1}{\partial x_1} &-\frac{\partial y_2}{\partial x_1} \\ &\frac{\partial y_2}{\partial x_1} &\frac{\partial y_1}{\partial x_1}
\end{pmatrix} = \frac 1{|\Phi'|^2} \begin{pmatrix}
&\frac{\partial x_1}{\partial y_1} &\frac{\partial x_2}{\partial y_1} \\ &-\frac{\partial x_2}{\partial y_1} &\frac{\partial x_1}{\partial y_1}
\end{pmatrix}
\end{align*}
and its transpose 
\begin{align*}
{D\Phi^{-1}}^t = \begin{pmatrix}
&\frac{\partial y_1}{\partial x_1} &\frac{\partial y_2}{\partial x_1} \\ &-\frac{\partial y_2}{\partial x_1} &\frac{\partial y_1}{\partial x_1}
\end{pmatrix} = \frac 1{|\Phi'|^2} \begin{pmatrix}
&\frac{\partial x_1}{\partial y_1} & -\frac{\partial x_2}{\partial y_1} \\ &\frac{\partial x_2}{\partial y_1} &\frac{\partial x_1}{\partial y_1}
\end{pmatrix}
\end{align*}
we will have
\begin{align*}
&\divv_{\bx} \hat \bv(\bx) = \frac{\partial y_1}{\partial x_1} \divv_{\bx} \bv + \frac{\partial y_2}{\partial x_1} \curl_{\bx} \bv = \\
&\frac 1{|\Phi'|^2} \left (
\frac{\partial y_1}{\partial x_1} \frac{\partial x_2}{\partial y_1}+
\frac{\partial y_2}{\partial x_1} \frac{\partial x_1}{\partial y_1} \right ) \hat w + 
\frac 1{|\Phi'|^2} \left (
\frac{\partial y_1}{\partial x_1} \frac{\partial x_2}{\partial y_2}+
\frac{\partial y_2}{\partial x_1} \frac{\partial x_2}{\partial y_1} \right ) \hat \rho=\\
&\left (
\left ( \frac{\partial y_1}{\partial x_1} \right )^2 +
\left ( \frac{\partial y_2}{\partial x_1} \right )^2 \right ) \hat \rho = |{\Phi^{-1}}'|^2 \hat \rho
\end{align*}
and
\begin{align*}
&\curl_{\bx} \hat \bv(\bx) =  \frac{\partial y_1}{\partial x_1} \curl_{\bx} \bv - \frac{\partial y_2}{\partial x_1} \divv_{\bx} \bv = \\
&\frac 1{|\Phi'|^2} \left (
\frac{\partial y_1}{\partial x_1} \frac{\partial x_1}{\partial y_1}-
\frac{\partial y_2}{\partial x_1} \frac{\partial x_2}{\partial y_1} \right ) \hat w+
\frac 1{|\Phi'|^2} \left (
-\frac{\partial y_1}{\partial x_1} \frac{\partial x_2}{\partial y_1}-
\frac{\partial y_2}{\partial x_1} \frac{\partial x_2}{\partial y_2} \right ) \hat \rho=\\
&\left (
\left ( \frac{\partial y_1}{\partial x_1} \right )^2 +
\left ( \frac{\partial y_2}{\partial x_1} \right )^2 \right ) \hat w = |{\Phi^{-1}}'|^2 \hat w.
\end{align*}

Then the system (\ref{intro:1})-(\ref{intro:4}) transfers to new one, defined in $B_{r_0}$:
\begin{eqnarray}
&&\divv~ \hat \bv(\bx) = |{\Phi^{-1}}'|^2 \hat \rho(\bx), \label{intro:freediv3}\\
&&\curl~ \hat \bv(\bx) = |{\Phi^{-1}}'|^2 \hat w(\bx),  \label{intro:curleq3}\\
&&\hat \bv(\bx)=\hat \boldg(\bx),~\vert \bx \vert=r_0,  \label{bound3}\\
&&\hat \bv(\bx)\to\bv_\infty,~\vert \bx \vert\to \infty.  \label{boundinf3}
\end{eqnarray}

Denote
\begin{align*}
q_k(r)=[|{\Phi^{-1}}'(z)|^2 \hat w(\bx)]_k, \\
r_k(r)=[|{\Phi^{-1}}'(z)|^2 \hat \rho(\bx)]_k,
\end{align*}
where $[\cdot]_k$ means the $k$-th Fourier coefficient and $z=re^{i\varphi}=x_1+ix_2$.

Rewrite (\ref{intro:freediv3}),(\ref{intro:curleq3}) in polar coordinates using Fourier coefficients $\hat v_{r,k}$, $\hat v_{\varphi,k}$:
\begin{eqnarray*}
&&{\frac {1}{r}}{\frac {\partial }{\partial r}}\left(r \hat v_{r,k}\right)+{\frac {ik}{r}} \hat  v_{\varphi,k} = r_k(r),\\
&&{\frac {1}{r}}{\frac {\partial }{\partial r}}\left(r \hat v_{\varphi,k}\right)-{\frac {ik}{r}} \hat  v_{r,k} = q_k(r).
\end{eqnarray*}

From (\ref{intro:zerocirculation}), (\ref{intro:zeroflux}), Stokes and divergence formulas we will have:
\begin{align*}
&\lim_{R\to\infty}\oint_{\vert \bx \vert=R}  \bv(\bx) \cdot d\mathbf{l} =
\oint_{\partial \Omega}   \boldg(\bx) \cdot d\mathbf{l} +  \int_{\Omega}  w(\bx) \dx =\\  &\oint_{\vert \bx \vert=r_0} \hat  \boldg(\bx) \cdot d\mathbf{l} + \int_{B_{r_0}} \frac {\hat w(\bx)}{\vert \Phi' \vert^2} \dx =  \\
&\oint_{\vert \bx \vert=r_0} \hat  \boldg(\bx) \cdot d\mathbf{l} +  2 \pi \int_{r_0}^\infty  \left [\frac {\hat w(\bx)}{\vert \Phi' \vert^2} \right ]_{k=0} s \ds = 0,\\
&\lim_{R\to\infty}\oint_{\vert \bx \vert=R} (  \bv(\bx),\bn) dl=
\oint_{\partial \Omega} (  \boldg(\bx),\bn) dl +  \int_{\Omega} \hat \rho(\bx) \dx = \\ &\oint_{S_{r_0}} (  \hat \boldg(\bx),\bn) dl + \int_{B_{r_0}} \frac {\hat \rho(\bx)}{\vert \Phi' \vert^2} \dx = \\ &\oint_{S_{r_0}} ( \hat  \boldg(\bx),\bn) dl +
2 \pi \int_{r_0}^\infty  \left [\frac {\hat \rho(\bx)}{\vert \Phi' \vert^2} \right ]_{k=0} s \ds = 0.
\end{align*}
Here $[\cdot]_{k=0}$ means zero Fourier coefficient.

The solution of the system (\ref{intro:freediv3})-(\ref{boundinf3}) is given by
\begin{eqnarray} \label{intro:BiotSavar1_2}
&&\hat v_{r,k} =   \frac{ir^{-\vert k \vert-1}}2 \int_{r_0}^r s^{\vert k \vert+1}( q_k - i r_k) \ds \\ 
&&+   \frac{ir^{\vert k \vert-1}}2 \int_r^\infty s^{-\vert k \vert+1}(  q_k + i r_k)\ds +  \alpha_k i r^{-k-1}   + v^\infty_{r,k}, \nonumber \\ \label{intro:BiotSavar2_2}
&&\hat v_{\varphi,k} = \frac{r^{-\vert k \vert-1}}2 \int_{r_0}^r s^{\vert k \vert+1}(q_k - i r_k)\ds \\ 
&&- \frac{r^{\vert k \vert-1}}2 \int_r^\infty s^{-\vert k \vert+1}(q_k + i r_k)\ds + \alpha_k   r^{-k-1} + v^\infty_{\varphi,k}.\nonumber 
\end{eqnarray}
with
$$
\alpha_k = \frac{r_0^{k+1}}{2}\left (\hat g_{\varphi,k}  - i \hat g_{r,k} \right ).
$$

%
%

Formulas (\ref{intro:BiotSavar1_2}), (\ref{intro:BiotSavar2_2}) with  condition (\ref{bound3}) lead to relations: 
\begin{align*}
r_0^{k-1}\int_{r_0}^\infty s^{-\vert k \vert+1} (  q_k + i r_k) \ds +  \hat g_{\varphi,k} +i \hat g_{r,k}= v^\infty_{\varphi,k}  + i v^\infty_{r,k},
\end{align*}
or 
\begin{align} 
&\int_{B_{r_0}} \frac {\hat w(\bx)+i\hat \rho(\bx)}{z^k} |{\Phi^{-1}}'(\bx)|^2 \dx 
+ \oint_{S_{r_0}} \frac{\hat \boldg(\bx)\cdot d\mathbf{l}}{z^k} 
+ i\oint_{S_{r_0}} \frac{(\hat \boldg(\bx),\bn)}{z^k} dl= \nonumber \\
 &\oint_{S_{r_0}} \frac{\bv_\infty\cdot d\mathbf{l}}{z^k} 
+ i\oint_{S_{r_0}} \frac{(\bv_\infty,\bn)}{z^k} dl \label{intconddiskconform}
\end{align}
for $k\in \NM \cup \{0\}$.

Returning to the original domain $\Omega$, since $| \hat \boldg | = |{\Phi^{-1}}'|\cdot |\boldg|$, and $dl$, $d\mathbf{l}$ should be multiply by $|\Phi'|$ under conformal map $\Phi$ from $\partial \Omega$ to $S_{r_0}$, then
finally these relations in $\Omega$  will take the form
\begin{align*} 
&\int_{\Omega} \frac {w(\bx)+i\rho(\bx)}{\Phi(z)^k}  \dx + \oint_{\partial \Omega} \frac{ \boldg(\bx)\cdot d\mathbf{l} + i( \boldg(\bx),\bn)dl}{\Phi(z)^k} = \nonumber \\
 &\oint_{\partial \Omega} \frac{\bv_\infty\cdot d\mathbf{l} + i(\bv_\infty,\bn) dl}{\Phi(z)^k},~k \in \NM \cup \{0\}.
\end{align*}

For horizontal flow $\bv_\infty=(v_\infty, 0)$ it will be the relations
\begin{align*}
& \int_{\Omega} \frac {w(\bx)+i\rho(\bx)}{\Phi(z)^k} d\bx + \oint_{\partial \Omega} \frac{\hat \boldg(\bx)\cdot d\mathbf{l}}{z^k} 
+ i\oint_{\partial \Omega} \frac{(\hat \boldg(\bx),\bn)}{z^k} dl= \\& \begin{cases} 0,~k\in \NM \cup \{0\},~k\neq 1, \\ 2 \pi i v_\infty,~k=1.
\end{cases}
\end{align*}

Formulas (\ref{intro:BiotSavar1_2}), (\ref{intro:BiotSavar2_2})  have an integral representation
\begin{align*}
\hat \bv(\bx)=\frac 1{2\pi} \int_{B_{r_0}}   \frac{(\bx-\by)^\perp}{\vert \bx-\by \vert ^2} |{\Phi^{-1}}'(\by)|^2 \hat w(\by)d\by  +  
\frac 1{2\pi} \int_{B_{r_0}}   \frac{\bx-\by}{\vert \bx-\by \vert ^2} |{\Phi^{-1}}'(\by)|^2 \hat \rho(\by)d\by + \\ 
\frac 1{2\pi} \oint_{S_{r_0}}   \frac{(\bx-\by)^\perp}{\vert \bx-\by \vert ^2} |{\Phi^{-1}}'(\by)| (\hat \boldg,\tau)dl  
\frac 1{2\pi} \oint_{S_{r_0}}   \frac{\bx-\by}{\vert \bx-\by \vert ^2} |{\Phi^{-1}}'(\by)|  (\hat \boldg,\bn)dl  +\bv_\infty.
\end{align*}
Here $\tau$ is a unit tangent vector to $\partial \Omega$ in a counter-clockwise direction.

In $\Omega$ this formula takes the form
\begin{align}
&\bv(\bx)=\frac 1{2\pi} D\Phi^t(\bx) \int_{\Omega}   \frac{(\Phi(\bx)-\Phi(\by))^\perp}{\vert \Phi(\bx)-\Phi(\by) \vert ^2} w(\by)d\by  +\nonumber\\
&\frac 1{2\pi} D\Phi^t(\bx) \int_{\Omega}   \frac{\Phi(\bx)-\Phi(\by)}{\vert \Phi(\bx)-\Phi(\by) \vert ^2} \rho(\by)d\by+\nonumber\\
&\frac 1{2\pi} D\Phi^t(\bx) \oint_{\partial \Omega}   \frac{(\Phi(\bx)-\Phi(\by))^\perp}{\vert \Phi(\bx)-\Phi(\by) \vert ^2} (\boldg,\tau)dl  + \nonumber \\
&\frac 1{2\pi} D\Phi^t(\bx) \oint_{\partial \Omega}   \frac{\Phi(\bx)-\Phi(\by)}{\vert \Phi(\bx)-\Phi(\by) \vert ^2} (\boldg,\bn)dl+\bv_\infty.   \label{intro:BS}
\end{align}

The kernels of the integrals involved in the above formula are the gradient and the skew gradient of the Green's function:
\begin{align*}
G(\bx,\by) = \frac 1{2\pi} \ln |\Phi(\bx)-\Phi(\by)|.
\end{align*}

\section{Uniqueness solvability of the planar div-curl problem}

\begin{theorem}Let $\Omega$ be an exterior domain with a Lipschitz boundary, such that there exists
a Riemann mapping $\Phi$ from $\Omega$ into $B_{r_0}$, which satisfies (\ref{intro:Phiprop}), (\ref{intro:Phiprop2}), and for $k \in \NM \cup \{0\}$ the following constraints on moments hold:
 \begin{align} 
&\int_{\Omega} \frac {w(\bx)+i\rho(\bx)}{\Phi(x_1+ix_2)^k}  \dx + \oint_{\partial \Omega} \frac{ \boldg(\bx)\cdot d\mathbf{l} + i( \boldg(\bx),\bn)dl}{\Phi(x_1+ix_2)^k} = \oint_{\partial \Omega} \frac{\bv_\infty\cdot d\mathbf{l} + i(\bv_\infty,\bn) dl}{\Phi(x_1+ix_2)^k}.\label{noslipcondintegralrieman}
\end{align} 

Then the following statements are true:
\begin{itemize}
\item if $\rho, w \in L_{2}(\Omega)$ then there exists the unique solution $\bv$ of the problem (\ref{intro:1})-(\ref{intro:4}). It is given by explicit formula (\ref{intro:BS}) and the estimate (\ref{nablaest}) holds.

\item if $\rho, w \in L_{2,N}(\Omega)$ with $N>1$ then 
\begin{align*}
\| \bv(\cdot) - \bv_\infty \|_{H^1(\Omega)} \leq C \left ( \| \rho \|_{L_{2,N}(\Omega)} + \| w \|_{L_{2,N}(\Omega)} + \|\boldg(\bx)\|_{H^{1/2}(\partial G)} \right ).
\end{align*}

\item if $\rho, w \in L_{2,N}(\Omega)$, $\nabla w \in L_{2}(\Omega)$ and no-slip condition $\boldg=0$ holds, then
\begin{align*}
&\|\bv(t,\cdot)\|_{L_\infty(\Omega)} \leq  \\ &~~~C\left ( \| \rho \|_{L_{2,N}(\Omega)} + \| w \|_{L_{2,N}(\Omega)} \right )^\frac14_{L_2(\Omega)} \times \|w(t,\cdot)\|_{L_2(\Omega)}^\frac 12  \|\nabla w(t,\cdot)\|^\frac 14 _{L_2(\Omega)} +  \|\bv_\infty\|_{\RM^2}.
\end{align*}
\end{itemize}

\end{theorem}

\begin{proof}
Uniqueness follows from the Hodge theory, since the boundary condition fixes the circulation and the flux around a solid. Existence follows from the explicit formula for the solution  (\ref{intro:BiotSavar1_2}), (\ref{intro:BiotSavar2_2}), or its integral representation (\ref{intro:BS}). 

From (\ref{noslipcondintegralrieman}) the relations (\ref{intconddiskconform}) are fulfiled, and then from the Theorem \ref{L2estthmdisk} follows the estimate of $\|\nabla \hat \bv (\bx)\|$ in $L_2(B_{r_0})$, and under Riemann mapping into $\Omega$ we will have the estimate of $\|\nabla \bv (\bx)\|$ in $L_2(\Omega)$.


Let's prove the $L_2$-estimate of the velocity. From the Theorem \ref{intro:propl2est} with help of (\ref{intro:Phiprop2}) with some $C_1$, $C_2$ holds: 
\begin{align*}
&\| \hat \bv(\cdot) - \bv_\infty \|_{H^1(B_{r_0})} \leq \\ &C_1 \Big ( \left \|
\rho(\Phi^{-1}(\cdot)) |{\Phi^{-1}}'(\cdot)|^2 \right \|_{L_{2,N}(B_{r_0})} + \left \|
 w(\Phi^{-1}(\cdot)) |{\Phi^{-1}}'(\cdot)|^2 \right \|_{L_{2,N}(B_{r_0})} +   \\ &\left \|\hat \boldg(\bx)|{\Phi^{-1}}'(\cdot)| \right \|_{H^{1/2}(S_{r_0})} \Big ) \leq C_2 \Big ( \| w(\cdot) \|_{L_{2,N}(\Omega)}
  + \| \rho(\cdot) \|_{L_{2,N} (\Omega)}  + \|\boldg(\bx)\|_{H^{1/2}(\partial \Omega)}\Big ).
\end{align*}

Then using the inequality with some $C_3$ 
\begin{align*}
\|  \bv(\cdot) - \bv_\infty \|_{H^1(\Omega)} \leq C_3 \| \hat \bv(\cdot) - \bv_\infty \|_{H^1(B_{r_0})}
\end{align*}
we obtain the estimate of $\| \bv(\cdot) - \bv_\infty \|_{H^1(\Omega)}$. 

Now we will prove the last estimate of the vector field $\bv(t,\cdot)$ in $L_\infty$ under no-slip condition $\boldg=0$. The following interpolation inequality is valid:
\begin{align*}
&\|\bv(t,\cdot) - \bv_\infty\|_{L_\infty(\Omega)} \leq C \|\bv(t,\cdot) - \bv_\infty\|^\frac 12 _{L_4(\Omega)} \|\nabla \bv(t,\cdot)\|^\frac 12_{L_4(\Omega)}.
\end{align*}

To bound the $L_4$-norm, we will use the estimate of the Calder\'on-Zygmund integral operator (\ref{bsest2}):
\begin{align*}
&\|\nabla \bv(t,\cdot)\|_{L_4(\Omega)} \leq C \|w(t,\cdot) \|_{L_4(\Omega)}.
\end{align*}

Then in view of the inequality
\begin{align*}
&\|w(t,\cdot) \|_{L_4(\Omega)} \leq C \|w(t,\cdot)\|^\frac 12_{L_2(\Omega)} \|\nabla w(t,\cdot)\|^\frac 12_{L_2(\Omega)}
\end{align*}
we will have
\begin{align*}
&\|\nabla \bv(t,\cdot)\|_{L_4(\Omega)} \leq C \|w(t,\cdot)\|^\frac 12_{L_2(\Omega)} \|\nabla w(t,\cdot)\|^\frac 12_{L_2(\Omega)}.
\end{align*}

And the $L_4$-norm of the vector field is controlled by the following $L_2$-norms:
\begin{align*}
\|\bv(t,\cdot) - \bv_\infty\|^4_{L_4(\Omega)} \leq  C \|\bv(t,\cdot)- \bv_\infty\|^2_{L_2(\Omega)} \|\nabla \bv(t,\cdot)\|^2_{L_2(\Omega)}.
\end{align*}

Then
\begin{align*}
&\|\bv(t,\cdot) - \bv_\infty\|_{L_\infty(\Omega)} \leq C \|\bv(t,\cdot)- \bv_\infty\|^\frac14_{L_2(\Omega)} \|w(t,\cdot)\|_{L_2(\Omega)}^\frac 12  \|\nabla w(t,\cdot)\|^\frac 14 _{L_2(\Omega)}.
\end{align*}

The Theorem is proved.\end{proof}

\section{Numerical results}
Here we will present the results of the numerical simulation. In our experiment we consider exterior div-curl problem for solenoidal flows with no-slip condition on the solid and horizontal flow at infinity $(v_\infty,0)$:
\begin{eqnarray}
&&{\rm div}~ \bv(\bx) = 0, \label{num:1}\\
&&{\rm curl}~ \bv(\bx) = w(\bx), \label{num:2}\\
&&\bv(\bx)=0,~\bx\in\partial\Omega,\label{num:3}\\
&&\bv(\bx)\to (v_\infty,0),~|\bx|\to \infty. \label{num:4}
\end{eqnarray}

Since the no-slip condition (\ref{num:3}) is fullfiled, then there exists a stream function $\psi$
$$\bv = \nabla^\perp \psi,$$ 
where $\nabla_\bx^\perp = (-\partial_{x_2}, \partial_{x_1})$.

The condition (\ref{num:3}) is a particular case of a slip interaction between solid and fluid, which is expressed by the formula: 
\begin{align}\label{num:slip}
(\bv, \bn) = 0,~\bx \in \partial \Omega.
\end{align}

In terms of stream function it can be rewritten as
\begin{align}
\psi(\bx) \equiv const,~\bx \in \partial \Omega. \label{num:noslipstream}
\end{align}

Also no-slip condition leads to Neumann boundary ones:
\begin{align}\label{num:neumann}
\frac {\partial \psi (\bx)}{\partial \bn} =0,~~\bx \in \partial \Omega.
\end{align}

At infinity horizontal flow $(v_\infty,0)$ transfers to boundary constraint on $\psi$: 
\begin{align} \label{num:infcond}
\psi (\bx) \to -v_\infty\cdot x_2,~|\bx|\to \infty.
\end{align}

Applying the curl operator to (\ref{num:2}) it goes to Poisson equation:
\begin{align}\label{num:poisson}
\Delta \psi = \curl \bv.
\end{align}

But the exterior Poisson problem with boundary condition (\ref{num:noslipstream}), (\ref{num:neumann}), (\ref{num:infcond}) is overdetermined and incorrect. From (\ref{num:3}) follows the condition of the zero-circulation around boundary:
\begin{align}\label{num:zerocircularity}
\oint_{\partial \Omega} \bv \cdot d\mathbf{l} = 0.
\end{align}

This condition combined with slip constraint (\ref{num:slip}) guarantees the correctness of the exterior div-curl problem (\ref{num:1}),(\ref{num:2}),(\ref{num:4}). And so (\ref{num:noslipstream}), (\ref{num:infcond}), (\ref{num:zerocircularity}) guarantee the correctness of the exterior Poisson problem (\ref{num:poisson}) which can be solved numerically by means of the finite element method (FEM). Conditions (\ref{num:noslipstream}) and (\ref{num:zerocircularity}) generate $N$ boundary relations, where $N$ is the amount of boundary nodes in a triangulated domain. An unbounded domain is restricted by a rectangular ones where we set the condition on infinity (\ref{num:infcond}). This condition also generates relations by the number of nodes in the rectangular domain. So, FEM becomes a closed computational method for the system (\ref{num:noslipstream}), (\ref{num:infcond}), (\ref{num:poisson}), (\ref{num:zerocircularity}).

In the numerical experiment Poisson equation was solved by FEM with continuous piecewise linear basis functions on triangles in a triangulated plane region with about 3600 nodes. Data for the vorticity function $w$ on the right-hand side of (\ref{num:2}) which satisfies the finite set of the integral  restrictions (\ref{noslipcondintegralrieman}) for $k<20$ with $\rho \equiv 0$ was generated in the AGVortex program from vortical flow with no-slip condition. In Figure~\ref{fig1} there is presented the solution of the div-curl problem for flows around a rectangle with vorticity, which is illustrated in
 Figure~\ref{fig2}. Figures~\ref{fig3}, \ref{fig4} demonstrate numerical solution of the same problem for flow around line segment.

\section{Discussion}
The velocity field vanishes on the boundary in some approximation (Figures~\ref{fig1}, \ref{fig3}) since in the computational process only a finite set of relations (\ref{noslipcondintegralrieman}) can be imposed ($k<20$ in the presented experiment), and the boundary velocity may produce high-frequency fluctuations with small amplitude corresponding to non-zero Fourier coefficients for large indices $k$. 

Regarding the geometry of the exterior domain $\Omega$, in the current research its boundary consists of two connected components (including point at infinity). But the presented approach can be extended to the case of more than two connected components, when the fluid flows through several solids. 

In exterior of the line segment the statement of the div-curl problem has its own peculiarities. In that case the line segment has "two sides" in some sense, and no-slip condition must be imposed on both sides of it. It can be considered as a limiting case of ellipses when its thickness goes to zero. Another way the definition of the solution can be obtained via Riemann mapping of the div-curl problem into the exterior of the disk, where we can correctly define the solution of the equation. 

Vortical flow in Figures~\ref{fig2}, \ref{fig4} breaks down into two parts with clockwise rotation (red color) and counter clockwise (blue color). These regions are separated by zero-vorticity curve ($w = 0$). This curve can pass through a wall of the line segment (see Figure \ref{fig4}) by reasons of symmetry.

Orthogonality relations (\ref{noslipcondintegralrieman}) can be extended to 3D domains. For example, in the case of the spherical solid, they can be written using vector spherical harmonics. No-slip condition implies zero value of the curl's normal component $(w,\bn) = 0$ on the boundary when the tangent components of $w$ must satisfy integral relations over the exterior of the sphere. 


\section{Conclusion}
Integral constraints (\ref{noslipcondintegralrieman})  imposed on vorticity function $w$,  divergence $\rho$ and boundary function $\boldg$ give unique
solvability of the exterior div-curl. It can be solved numerically   by reducing to Poisson problem. This approach can be extended to flows around several streamlined solids.

For solenoidal flows no-slip condition generates a series of integral constraints on vorticity (orthogonality relations). Under these restrictions the div-curl problem can be numerically solved with more typical boundary conditions - slip condition combined with zero circulation. The source of vorticity data, which satisfies such integral constraints, can be derived by solving any equations of motion written in vorticity form that describe fluid dynamics - Stokes, Navier-Stokes, and Boussinesq equations. And the div-curl problem becomes the crucial step of the simulation procedure in the planar computational fluid dynamics. 

\section{Data availability statements}
The experimental data and the simulation results that support 
the findings of this study are available in Figshare with the identifier \\ https://doi.org/10.6084/m9.figshare.30028255

\section{Statements and Declarations}
No funding was received for conducting this study.

\newpage

\begin{figure}[h!]
\centering
\includegraphics[width=0.9\textwidth]{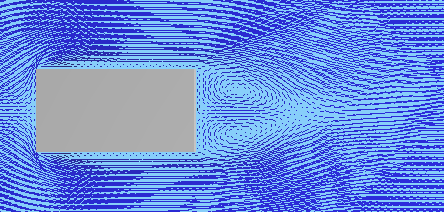}
\caption{Velocity field around rectangle. Flow data was generated in program AGVortex.}\label{fig1}
\end{figure} 
 
\begin{figure}[h!]
\centering
\includegraphics[width=0.9\textwidth]{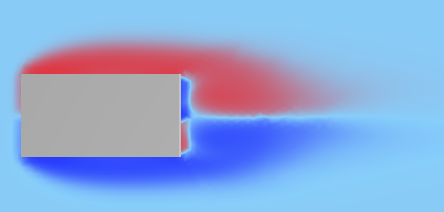}
\caption{Vortical flow around rectangle. Flow data was generated in program AGVortex.  }\label{fig2}
\end{figure}

\begin{figure}[h!]
\centering
\includegraphics[width=0.9\textwidth]{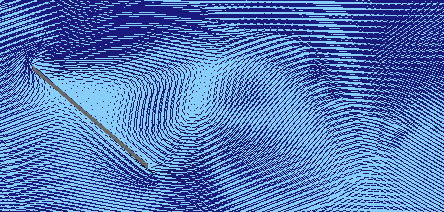}
\caption{Velocity field around line segment. Flow data was generated in program AGVortex. }\label{fig3}
\end{figure}

\begin{figure}[h!]
\centering
\includegraphics[width=0.9\textwidth]{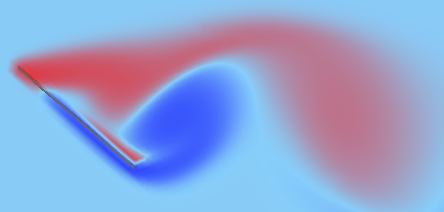}
\caption{Vortical flow around line segment. Flow data was generated in program AGVortex.    }\label{fig4}
\end{figure}

\end{document}